\documentclass{amsart}

\usepackage{enumerate}
\usepackage{ifpdf}
\usepackage{amsmath}
\usepackage{amsfonts}
\usepackage{amssymb}
\usepackage{amsthm}
\usepackage[hyperfootnotes=false]{hyperref}
\usepackage{setspace}

\newtheorem{thm}{Theorem}[section]

\newtheorem{lemma}[thm]{Lemma}

\theoremstyle{definition}
\newtheorem{defn}[thm]{Definition}
\newtheorem{example}[thm]{Example}

\theoremstyle{remark}
\newtheorem{remark}[thm]{Remark}

\newcommand{\CC}{\mathbb C}

\newcommand{\Id}{\textrm{Id}}
\newcommand{\R}{\textrm{Re}}
\date{\today}

\ifpdf
\author{Liz Raquel Vivas}
\address{The Ohio State University\\
Columbus, OH\\
USA}
\email{vivas.3@osu.edu}

\title[Unstable manifolds for parabolic skew-products]{Parametrization of unstable manifolds for parabolic skew-products}

\begin{document}
\bibliographystyle{plain}

\begin{abstract}
Given a parabolic map in one dimension $f(z) = z + O(z^2)$, $f \neq \textrm{Id}$, it is known that there exists the analogous of stable and unstable domains. That is, domains in which every point is attracted by $f$ (or by the inverse $f^{-1}$) towards the fixed point. In this paper we prove that there exists a natural parametrization for the unstable manifold in terms of iterates for some subset of parabolic maps.
Furthermore, we prove that this parametrization is valid also in the case of skew-product maps that satisfy certain conditions. Finally, we give an application of this fact to construct Fatou disks for skew-product maps, in a similar way than in the paper \cite{PeVi2014}.
\end{abstract}

\maketitle

\section{Introduction}\label{section:intro}

In a recent joint paper with Peters \cite{PeVi2014}, we construct a skew product map $F$ that have one dimensional Fatou-disks. That is, disks on which the restriction of $F$ and its iterates form a normal family. In this way we continued the study of skew-product maps and its Fatou components on a neighborhood of a fixed fiber \cite{Lilov}. Our method was to use a parametrization theorem of skew product maps for which the map on the parameter fiber is attracting (similar than \cite{Hubbard}), and in the other coordinate is repelling. Even more, our maps were resonant. A natural question is whether we can construct a similar case for skew-product maps that are parabolic. In this paper we prove that this is indeed the case.

We will prove this theorem using a similar strategy than the proof of our theorem on \cite{PeVi2014}. First we prove a parametrization theorem for a class of parabolic maps. In the following section we prove that a similar parametrization theorem can be obtained for skew product maps. In the last section we construct skew product maps that are parabolic in each coordinate that have Fatou disks. We also prove that our Fatou disks cannot be enlarged to Fatou components. If we could enlarge them, then we would have a wandering Fatou component. \footnote{For a very recent result concerning Fatou wandering components see \cite{ABDPR}.}

\section{Parametrization of unstable manifolds of parabolic maps in one dimension}

Given a parabolic map $F(z) = z + a_kz^k+O(z^{k+1}), a_k \neq 0, k \geq 2$, the Leau-Fatou theorem say that there exist $k-1$ regions in which each point is attracted to the origin under iterates by $F$, and also $k-1$ regions in which the orbits are going towards the origin under the iterates of $F^{-1}$. We can think of these regions as stable and unstable manifolds of our map $F$. In each one of these regions is possible to find a Fatou coordinate. That is, a change of coordinates map such that $F$ is conjugate to a translation. However, in order to find the change of coordinates, there is not an iterative process as it is the case for hyperbolic maps.

However, for a certain class of maps the change of coordinates can be recovered by using iterations of our parabolic map.

\begin{thm}
Let
$$
f(z) = z + z^2 + z^3 + O(z^4).
$$
Define:
$$
\phi_n(z):= f^n\left(\frac{z}{1+nz}\right).
$$
Then $\phi_n$ converges to a map $\phi$ defined in $V_\epsilon = \{z, |z-\epsilon| < \epsilon\}$ a parametrization of the unstable manifold of $f$, such that:
\begin{equation}\label{paramone}
\phi\left(\frac{t}{1-t}\right)=f(\phi(t)).
\end{equation}
Even more, if $z \in V_\epsilon$, then:
$$
|\phi_{n+1}(z)-\phi_n(z)| < \frac{C'|z|^2}{(1+n|z|)^2} < \frac{C'}{n^2}.
$$
\end{thm}

\begin{proof}
We will prove the convergence by using classical results of complex analysis.
First, we change coordinates of $f$ to the infinity. Let $i(z) = 1/z$. Then the map $g=i\circ f \circ i$ is given by:
$$
g(w) = w - 1 + \eta(w) = w-1+O\left(\frac{1}{w^2}\right).
$$
We have then the following sequence:
$$
\psi_n(w) = g^n(w+n)
$$
is such that $\phi_n = i \circ \psi_n \circ i$. So proving convergence for $\phi_n$ is equivalent than proving convergence for $\psi_n$.
We consider the inverse map $h$ of $g$. Close to infinity we also have:
$$
h(u) = u + 1 + O(1/u^2).
$$
It is a classical result to prove that the sequence $\gamma_n(u) = h^n(u) - n$ is convergent for $\R(u) > R'$ with the condition:
$$
\gamma(u) +1 = \gamma(h(u)).
$$
Since $\gamma_n \circ \psi_n = \Id$, then we have that $\psi_n$ converges to $\psi$ on the image of $\gamma_n$, which contains a domain of the same form $\R(w)  > R$. We will have:
$$
\psi(w) = g \circ\psi(w +1).
$$
Therefore $\phi_n$ converges on the domain $\{z - 1/2R\} <1/2R = V_\epsilon$ for $\epsilon=1/2R$ and we will obtain: 
$$
\phi(z) =f \circ \phi \left(\frac{z}{1+z}\right),
$$
which is equivalent than \eqref{paramone}.

To obtain estimates on the speed at which the $\phi_n$ converges we also look at the speed and estimates for the limit function $\gamma$.
We have that:
$$
\gamma_n(u) = u + \sum O(\frac{1}{u_i^2}),
$$
and by the estimates we can see that for $\R u > R'$ so that $|h(u)-u-1|<1/2$, we will have:
$$
|\gamma_n(u) - u| < \frac{C}{|u|}  - \frac{C}{|u+n/2|}.
$$
Then the limit function:
$$
|\gamma(u) - u| < \frac{C}{|u|}.
$$
Since the inverse of the limit of $\gamma_n$ is the limit $\psi$ of the sequence $\psi_n$, we also have the estimate:
$$
|\psi(w) - w| < \frac{C}{|\psi(w)|} < \frac{C'}{|w|}. 
$$
And now using $\phi = i \circ \psi \circ i$, we obtain:
$$
|\phi(z) - z| < C'|z|^2|\phi(z)| < C''|z|^3. 
$$
Likewise we can make exactly the same estimates for the differences:
\begin{align*}
|\gamma_{n+1}(u) - \gamma_n(u)| <  \frac{C}{|h_n(u)|^2} \\
|u - \gamma^{-1}_{n+1}\circ \gamma_n(u)| < \frac{C'}{(|u|+n)^2}\\
|\psi_{n}\circ\psi^{-1}_{n+1}(w) - w| < \frac{C''}{(|w|+n)^2}\\
|i\circ \phi_{n}\circ\phi^{-1}_{n+1}(z) - \frac{1}{z}| < \frac{C''|z|^2}{(1+n|z|)^2}\\
|\phi_{n}\circ\phi^{-1}_{n+1}(z) -z| < \frac{C'''|z|^2}{(1+n|z|)^2}.
\end{align*}
Now we rename the constant and can easily check that:
\begin{equation}\label{eqdifference}
|\phi_{n+1}(z) -\phi_n(z)| < \frac{C|z|^2}{(1+n|z|)^2}.
\end{equation}
This concludes the proof of the theorem.
\end{proof}

We will give a name to maps of this form.
\begin{defn}
We say $f$ is $\textit{special}$ if $f$ is as defined in Theorem 2.1.
\end{defn}

\begin{remark}
In fact, if $f$ is of the following form:
$$
f(z) = z + \alpha z + \alpha^2 z^2 + O(z^3),
$$
where $\alpha \in \CC$, we also have an iterative process for finding the parametrization of the unstable manifold of $f$.
Our special maps are the ones for which $\alpha=1$.
\end{remark}
\begin{remark}
For any $k>0$, we can choose the higher orders of $f$ correctly, so that 
$$
|\phi(z) - \phi_n(z)| < O(1/n^k) 
$$
for $|z|$ bounded.
\end{remark}

\section{Skew parabolic maps and parametrization}

In this section we prove that the same parametrization works for a skew product type of map $(t,z) \to (g(t),f_t(z))$ under the assumption that $f_t$ is special, that is of the form above for each $t$.
	
We start with any map of the form:
$$
F(t,z) = \left(\frac{t}{1+t}, f_t(z)\right),
$$
where each $f_t$ is special.
Then we prove:
\begin{thm}
Define:
$$
\phi_n(t)  = \pi_2 F^{n}\left(\frac{t}{1+(n+1)t},\frac{t}{1+nt}\right)
$$
converges for $t$ in a domain of type $V_\epsilon$ as above.

Even more, the limit $\phi$ is the parametrization of the unstable manifold of $f_0$, ie.
\begin{align}\label{paramskew}
\phi\left(\frac{t}{1-t}\right) = f_0(\phi(t))
\end{align}
\end{thm}

Once again, it is easier to prove this, after we change coordinates to infinity, now in both variables. Let $(u,w) = I(t,z) = (1/t,1/z)$. In these new variables:
$$
G(u,w) = I \circ F \circ I = \left(u+1, g_u(w)\right),
$$
where $1/g_u(w) =f_{1/u}(1/w)$ or $i\circ g_u\circ i=f_{1/u}$.
We want to prove that the sequence:
$$
\phi_n(t)  = f_{t_{n}}\ldots f_{t_2}f_{t_1}\left(\frac{t}{1+nt}\right)
$$
converges, where $t_i = t/(1+(n+i)t)$.
Going to infinity:
$$
i \circ \phi_n( i \circ i t)  = i \circ f_{t_{n}}\ldots f_{t_2}f_{t_1} i \circ i\left(\frac{t}{1+nt}\right),
$$
where $u=1/t$, then
$$
i \circ \phi_n\circ i(u)  = g_{1/t_{n}}\ldots g_{1/t_2}g_{1/t_1}(u+n)
$$
and each $u_i:=1/t_i = (n+i) + u$, we need to prove that the sequence:
$$
\psi_n(u)  = g_{u+2n}\ldots g_{u+n+2}g_{u+n+1}(u+n)
$$
converges.
Notice that $g_u(w) = w - 1 +\theta_u(w) = w - 1 +O_u(1/w^2)$, that is, for $|w|>R$ and $|u|>R$, then $|\theta_u(w)| < A/|w|^2$, where $A$ is fixed.

We first choose $R$ large enough so that for $|w|>R$, then $|\theta_u(w)| < 1/10$. So we will clearly have that for any $u$, $\R (g_u(w)) > \R(w)$.
Therefore the domain $W  = \R(w) > R$ is invariant by $g_u$. In this domain we also have the easy estimates:

\begin{lemma}  For $w\in W$,
$$
|w|+9/10< |g_u(w)| <|w|+11/10,
$$
and therefore:
$$
|w|+9k/10< |g^k_u(w)| <|w|+11k/10.
$$
\end{lemma}

Let us call:
$$
u_{n,i} = g_{u+n+i} \ldots g_{u+n+1}(u+n), 1 \leq i \leq n,
$$
and:
$$
u_{n,0} = u+n.
$$

We choose $R' > R + A/R$, then we can prove that:
\begin{lemma} If $\R(u) > R'$, then:
$$
\R(u_{n,i}) > R+n-i
$$
for $0\leq i\leq n$.
\end{lemma}
\begin{proof}
For $i=0$ is trivial. Assume by induction it is valid for $i\leq k$. 
By definition:
$$
u_{n,i+1} = u_{n,i}-1+\theta_{u+n+i+1}(u_{n,i}).
$$
Then
\begin{align}
\sum_{i=0}^k(u_{n,i+1} - u_{n,i}) &= \sum_{i=0}^k\left(-1+\theta_{u+n+i+1}(u_{n,i})\right)\\
u_{n,k+1} &= u_{n,0} -(k+1)+ \sum_{i=0}^k \theta_{u+n+i+1}(u_{n,i})
\end{align}
And by induction up to $k$, each:
$$
|\theta_{u+n+i+1}(u_{n,i})| < \frac{A}{|u_{n,i}|^2} < \frac{A}{(R+n-i)^2}
$$
$$
\sum_{i=0}^{k}|\theta_{u+n+i+1}(u_{n,i})| < \sum_{i=0}^{k} \frac{A}{(R+n-i)^2} < \frac{A}{R},
$$
for any $k\leq n$.
Then in equation (4):
\begin{align*}
\R(u_{n,k+1}) &= \R(u_{n,0}) -(k+1)+ \R(\sum_{i=0}^k \theta_{u+n+i+1}(u_{n,i}))\\
&\geq \R(u_{n,0}) -(k+1) - \sum_{i=0}^k |\theta_{u+n+i+1}(u_{n,i})|\\
&> R' + n -(k+1) - \frac{A}{R}\\
&> R + n - (k+1).
\end{align*}
And the lemma is proven.
\end{proof}
Recall that $|\theta_u(w)| = |g_u(w) - w + 1| = |O_u(1/w^2)| < A/|w|^2$ for $|w| > R, |u|>R$.
\begin{lemma}
Let $R<S<|x|,|y|$ and $|u|, |v| >R$. Then:
\begin{align} 
|\theta_u(x) - \theta_v(y)| \leq |x-y|\left(\frac{2A}{S^3}\right) + |u-v|\frac{A}{|u||v|S^2}
\end{align}
\end{lemma}

\begin{proof}
\begin{align*}
|\theta_u(x) - \theta_v(y)| &\leq |\theta_u(x) - \theta_u(y)| + |\theta_u(y) - \theta_v(y)| \\
&\leq |\theta_u(x) - \theta_u(y)| + |\theta_u(y) - \theta_v(y)|
\end{align*}
and we can see that:
$$
|\theta_u(x) - \theta_u(y)| < |x-y|\left(\frac{2A}{|x|^3}\right)
$$
and
$$
|\theta_u(y) - \theta_v(y)| < |u-v|\frac{A}{|u||v||x|^2}.
$$
\end{proof}

\begin{lemma}
For $0 \leq k \leq n$, let:
\begin{align*}
|u_{n+1,k+1} - u_{n,k}|\leq C_k
\end{align*}
Then
$$
C_n < \frac{4A}{(R+n)^2}.
$$
\end{lemma}
\begin{proof}
For $k=0$:
\begin{align*}
|u_{n+1,1} - u_{n,0}| = |g_{u+n+2}(u+n+1) - u-n| = |\theta_{u+n+2}(u+n+1)| < \frac{A}{|R+n+1|^2}.
\end{align*}
Assume it is valid for  $i=k-1$:
\begin{align*}
|u_{n+1,k} - u_{n,k-1}|\leq C_{k-1}
\end{align*}
then for $i=k$:
\begin{align*}
|u_{n+1,k+1} - u_{n,k}| &= \left| g_{u+n+k+2}(u_{n+1,k}) - g_{u+n+k}(u_{n,k-1}) \right| \\
&= \left| u_{n+1,k} + \theta_{u+n+k+2}(u_{n+1,k}) - u_{n,k-1} - \theta_{u+n+k}(u_{n,k-1}) \right| \\
&= \left| u_{n+1,k} - u_{n,k-1}\right| + \left|\theta_{u+n+k+2}(u_{n+1,k})  - \theta_{u+n+k}(u_{n,k-1}) \right| \\
&\leq C_{k-1} + \left|\theta_{u+n+k+2}(u_{n+1,k})  - \theta_{u+n+k}(u_{n,k-1}) \right| \\
\end{align*}
Recall that $|u_{n+1,k}|$ and $|u_{n,k-1}| >  R + n + 1 - k =S >R$ and the subindices of $\theta$ are also larger in modulus than $R$.
Then we apply the lemma:
\begin{align*}
|\theta_u(x) - \theta_v(y)| &\leq |x-y|\left(\frac{2A}{S^3}\right) + |u-v|\frac{A}{|u||v|S^2}\\
\left|\theta_{u+n+k+2}(u_{n+1,k})  - \theta_{u+n+k}(u_{n,k-1}) \right|  &\leq |u_{n+1,k}-u_{n,k-1}|\left(\frac{2A}{(R+n+1-k)^3}\right)  + \\
&+\frac{2A}{|u+n+k+2||u+n+k|(R+n+1-k)^2},\\
\left|\theta_{u+n+k+2}(u_{n+1,k})  - \theta_{u+n+k}(u_{n,k-1}) \right|  &< |u_{n+1,k}-u_{n,k-1}|\left(\frac{2A}{(R+n+1-k)^3}\right)  + \\
&+\frac{2A}{(R+n+k)^2(R+n-k)^2}.
\end{align*}
Then back above:
\begin{align*}
|u_{n+1,k+1} - u_{n,k}| &< C_{k-1} + \left|\theta_{u+n+k+2}(u_{n+1,k})  - \theta_{u+n+k}(u_{n,k-1}) \right| \\
&< C_{k-1}\left(1 + \frac{2A}{(R+n+1-k)^3}\right) +\frac{2A}{(R+n+k)^2(R+n-k)^2},\\
\end{align*}
So we obtain: 
$$
C_k=C_{k-1}\left(1 + \frac{2A}{(R+n+1-k)^3}\right) +\frac{2A}{(R+n+k)^2(R+n-k)^2}.
$$
We can easily see that:
$$
\prod_{k=1}^n\left(1 + \frac{2A}{(R+n+1-k)^3}\right)  \sim \exp \sum_{k=1}^n  \frac{2A}{(R+n+1-k)^3} << \exp(\frac{A}{R^2})
$$
and we can choose $R$ large so $\exp(\frac{A}{R^2})<2$. Each of the partial products will be also bounded by $2$ then and therefore we have:
$$
C_n \leq 2C_{0} + \sum_{k=1}^n \frac{4A}{(R+n+k)^2(R+n-k)^2}.
$$
An easy computation shows that 
$$
\sum_{k=1}^n\frac{1}{(R+n+k)^2(R+n-k)^2} \sim  \frac{1}{(R+n)^3}\ln\left(\frac{R+2n}{R}\right) + \frac{1}{4(R+n)^2}\left(\frac{1}{R} - \frac{1}{R+2n}\right).
$$
All the terms on the right are of order $1/n^2$, and we can bound each of them, therefore we obtain:
$$
C_n < \frac{4A}{(R+n)^2}.
$$
\end{proof}

\begin{proof}[Proof of Theorem 3.1]
From the last lemma:
$$
|\Psi_{n+1}(u) - \Psi_{n}(u)| < \frac{4A}{(R+n)^2}.
$$
It is easy to go back to the $z$ coordinate and obtain the following estimate:
$$
|1/\phi_{n+1}(t) - 1/\phi_{n}(t)| < \frac{4A}{(R+n)^2}.
$$
$$
|\phi_{n+1}(t) - \phi_{n}(t)| < \frac{A'}{(1+n\epsilon)^2}
$$
where $\epsilon = 1/R$,$t \in V_\epsilon = \{|t-\epsilon/2|<\epsilon/2\}$ and $A'$ is bounded by $A, \epsilon$ and the supremum of $\phi$ in $V_\epsilon$. 

We have the immediate estimate
\begin{align}
|\phi(t) - \phi_n(t)|  < \frac{A'}{1+n\epsilon}
\end{align}
for $t \in  V_\epsilon$. This concludes the proof of Theorem 3.1.
\end{proof}

So far, our map $\phi$ is only defined in the domain $V_\epsilon$. We can extend $\phi$ to all of $\CC$ by using the functional equation \eqref{paramskew}.
Let $t\in \CC, t\neq 0$, for any $t$ there exists $N=N(t)$ such that $\frac{t}{1+Nt} \in V_\epsilon$. We define:
$$
\phi(t) := f_0^N\left(\phi(\frac{t}{1+Nt})\right).
$$
It is easy to prove that $\phi(t)$ can also be defined in terms of the iterates of $F$, as follows. Define:
\begin{align*}
\phi_{N+n}(t) := \pi_2F^{N}\left(\frac{t}{1+(N+2n+1)t},\phi_n(\frac{t}{1+Nt})\right), n \geq 0.
\end{align*}
All the terms inside the parenthesis in the right hand side are well defined and we have good estimates for the differences for $\phi_n$ and $\phi_{n+1}$.
Then we have:
$$
\phi(t) = \lim_{n \to \infty} \phi_{N+n}(t).
$$
Then, for any given $t \in \CC$, there exists a $N=N(t)$, and a constant $C=C(N)$ such that:
\begin{align}
\|\phi_{n+1}(t) - \phi_n(t)\| < C\frac{A}{(1+(n-N)\epsilon)^2},
\end{align}
\begin{align}
\|\phi_{n+1}(t) - \phi(t)\| < C\frac{A}{1+(n-N)\epsilon},
\end{align}
for any $n \geq N+1$.

It is well known that the range of $\phi$ is the whole complex plane $\CC$. See \cite{Milnor} for a proof.

\section{Application: Skew parabolic maps with Fatou disks}

We give an application of the theorem above to dynamics. We will prove a theorem similar than in \cite{PeVi2014}, that is, we prove that there exists some skew product parabolic maps that have Fatou disks. Our construction however, does not allow to fatten the disks. We prove that statement at the end of this section.

\begin{defn} Let:
$$
F(t,z) = \left(\frac{t}{1+t}, f_t(z)\right).
$$
We say $F$ is \textit{special} if the following conditions are satisfied:
\begin{itemize}
\item  $f_t$ is special for any fixed $t$ ($f_t(0) = 0$).
\item  $f_t(x_0) = t$ for some $0\neq x_0 \in \CC$. 
\end{itemize}
\end{defn}

Then we define:
\begin{thm}
For any fixed $t$ we will have that the following iterates:
$$
\phi_n(t) = \pi_2 F^{n+1}\left(\frac{t}{1+nt}, x_0\right).
$$ 
Then the $\phi_n$ converges to the parametrization of the unstable manifold of $f_0$, as above:
\begin{align}\label{param}
\phi\left(\frac{t}{1-t}\right)=f_0(\phi(t)).
\end{align}
\end{thm}

\begin{proof}
We easily see that:
$$
F\left(\frac{t}{1+nt}, x_0\right) = \left(\frac{t}{1+(n+1)t}, \frac{t}{1+nt}\right),
$$
and now we apply Theorem 3.1.
\end{proof}

Now, let $F$ as before $F(t,z) = \left(\frac{t}{1+t}, f_t(z)\right)$, with the additional conditions:
\begin{itemize}
\item  $x_0$ a critical point of $f_0$ of order at least 4. 
\end{itemize}

It is possible to construct a map as desired.  For example:
\begin{example}
$$
F(t,z) = \left(\frac{t}{1+t}, (z+1)^4(z-3z^2+7z^3)+t(1+(z+1)^4(-1+4z-10z^2+20z^3))\right).
$$
We have $f_t(z) =  z + z^2 + z^3 + O(z^4,z^4t)$.
Then $x_0=-1, f_0(x_0) = 0$ satisfies all the requirements.
\end{example}

Now, let $t_0 \in \mathbb C$ be such that $\phi(t_0) = x_0$. 

We will refer to the complex lines $\{t = \frac{t_0}{1+nt_0}\}$ as \emph{critical fibers}.

\begin{defn}\label{def:disks}
We define the vertical disks $D_n$ as follows:
\begin{align*}
D_n := \left\{ \left(\frac{t_0}{1+nt_0}, z\right) \mid |z - x_0| <  n^{-3/4} \right\}.
\end{align*}
\end{defn}

We will prove that for $n$ sufficiently large, the forward orbits of the disks $D_n$ all avoid the bulged Fatou components of $F$.

Note that $t_0$ might not be in our domain $V_\epsilon$ above. However, for a fixed $N'=N(t_0)$, we do have that $\frac{t_0}{1+Nt_0}$ is in $V_\epsilon$, for $N>N'$.
Therefore we can obtain all the estimates for iterates of $F$ after we iterate $F$, $N'$ times first.

We will need the following lemma:
\begin{lemma}
Let $(t,z)$ and $(t,w)$, be such that $t,z \in V_\epsilon$ and $|w-z| < \frac{C}{n^3}$. Then:
\begin{align*}
|\pi_2(F^n(t,z) -F^n(t,w))| < Cn^2|z-w|,
\end{align*}
where $C$ is fixed independent of $n$.
\end{lemma}

\begin{proof}
If $z=0$, then we have:
\begin{align*}
|\pi_2F^n(t,w))| = |f_{t_n}\circ\ldots f_{t}(w)|
\end{align*}
and recall that each $f_t$ is of the form $f_t(z) = z + z^2 + O(z^3)$. Let $w_{k+1} = f_{t_k}(w_k)$ and $w_0=w$. Then we have that at infinity: 
$$
\frac{1}{w_1} =\frac{1}{w} -1+O(1/w).
$$
Our initial $|w| < C/n^3$, then $|1/w| > n^3/C$, so when we apply $n$ times our map $f_t$, we have:
$|1/w_n| > |1/w_0| - n(1+K)$.
Going back to the original coordinates:
$$
|w_n| < \frac{1}{|1/w_0| - n(1+K)}  = \frac{|w_0|}{1 - n|w_0|(1+K)} < Cn^2|w_0|.
$$
When $z \neq 0$ then there exists $n$ large so that $|w-z| < \frac{C}{n^3}$ will imply $w \in V_\epsilon$. Then we have much stronger estimates as in Lemma 3.5.
\end{proof}

\begin{thm}\label{thm:nested}
For $n$ sufficiently large we have that
$$
F^{n+1}(D_n) \subset D_{2n+1}.
$$
\end{thm}
\begin{proof}
For the center:
$$
F^{n+1}\left(\frac{t_0}{1+nt_0},x_0 \right)  = \left(\frac{t_0}{1+(2n+1)t_0},\phi_n(t_0)\right)
$$
and from lemma 3.4 we have:
$$
Cn^{-1} = |\phi_n(t_0) - x_0| < r(2n+1) = (2n+1)^{-3/4}.
$$
Since $x_0$ is a critical point of order at least $3$, then:
$$
F\left(\frac{t_0}{1+nt_0},x \right) - F\left(\frac{t_0}{1+nt_0},x_0 \right) = (0, (x+1)^4C),
$$
where $C$ is bounded independent of $n$.
Then when we apply $n$ more iterates:
For the rest of the disk:
$$
\left|\pi_2\left(F^{n+1}\left(\frac{t_0}{1+nt_0},x_0 + \rho \right) - F^{n+1}\left(\frac{t_0}{1+nt_0},x_0 \right) \right)\right| \leq Cn^2|\rho|
$$
We use the lemma  5.3 above, and therefore for $n$ large:
$$
Cn^{-1}=Cr(n)^4n^2 < r(2n+1) = (2n+1)^{-3/4}.
$$
\end{proof}

\begin{remark}
An immediate consequence of Theorem \ref{thm:nested} is that for sufficiently large $n \in \mathbb N$ there exists a sequence of $l_n \to \infty$ so that $F^{l_n}(D_n) \rightarrow (0,x_0)$ as $\ell
\rightarrow \infty$.
\end{remark}

However, we can easily that we cannot enlarge our one dimensional disks into domains. Given any point $(t_n,z_n) \in D_n$, then if there exists a $B_n$ open such that $D_n \subset B_n$, we will have that there exists $(s,x_0) \in D_n$ where $|s-t_n| < \epsilon_n$. If $B_n$ is to be mapped to another open domain $B_{2n+1}$ by $F^n$, then we will have that $F^{n+1}(s,x_0) \in B_{2n+1}$. Therefore:
\begin{align*}
|\pi_2 F^{n+1}(s,x_0) - x_0| < (2n+1)^{-3/4}
\end{align*}
Then we will have:
$$
\pi_2 F^{n+1}(s,x_0) = \pi_2 F^n\left(\frac{s}{1+s},s\right) = \phi_n(s')
$$
where $\frac{s'}{1+ns'} = s$ ie. $s' = \frac{s}{1-ns}$. Since $\epsilon_n > |s-t_n| = |\frac{s'}{1+ns'} - \frac{t_0}{1+nt_0}| = \frac{|s'-t_0|}{|(1+ns')(1+nt_0)|}$.
Therefore:
\begin{align*}
|\pi_2 F^{n+1}(s,x_0) - \phi_n(x_0)| = |\phi_n(s')- \phi_n(x_0)|\geq C|s'-x_0|
\end{align*}
and the last term is a bounded number away from zero. Therefore it is not possible to fatten the disks above.

\end{document}